\documentclass[11pt]{amsart}

\usepackage[utf8x]{inputenc} 
\usepackage{enumerate} 
\usepackage{amsmath} 
\usepackage{amssymb}
\usepackage{color}
\usepackage{multirow}
\usepackage[pagebackref]{hyperref} 
\usepackage{verbatim}
\usepackage[normalem]{ulem}
\usepackage{comment}
\definecolor{Gray}{gray}{0.9} 
\usepackage[color, arrow,curve,matrix,tips,frame]{xy}
\usepackage{graphicx}

\theoremstyle{plain} 
\newtheorem{proposition}{Proposition}[section] 
\newtheorem{theorem}[proposition]{Theorem} 
\newtheorem{lemma}[proposition]{Lemma}

\newtheorem{question}[proposition]{Question} 

\theoremstyle{definition} 
\newtheorem{definition}[proposition]{Definition} 
\newtheorem{example}[proposition]{Example} 
\newtheorem{setting}[proposition]{Setting} 

\theoremstyle{remark} 
\newtheorem{remark}[proposition]{Remark}

\newcommand{\QQ}{{\mathbb{Q}}} 

\newcommand{\ZZ}{{\mathbb{Z}}} 
\newcommand{\CC}{{\mathbb{C}}} 
\newcommand{\PP}{{\mathbb{P}}} 
\newcommand{\RR}{{\mathbb{R}}}

\newcommand{\cO}{{\mathcal{O}}}

\providecommand{\vol}{{\rm vol}} 
\providecommand{\BDPP}{\mathop{\rm BDPP}}

\newcommand{\floor}[1]{\left\lfloor #1 \right\rfloor}
\newcommand{\abs}[1]{\left\lvert #1 \right\rvert}

\DeclareMathOperator{\Amp}{Amp}
\DeclareMathOperator{\Nef}{Nef}
\DeclareMathOperator{\Eff}{Eff}
\DeclareMathOperator{\Bir}{Bir}
\DeclareMathOperator{\BigC}{Big}
\DeclareMathOperator{\Aut}{Aut}

\DeclareMathOperator{\id}{id}
\DeclareMathOperator{\MovC}{\overline{Mov}}
\DeclareMathOperator{\EffC}{\overline{Eff}}
\DeclareMathOperator{\Mov}{Mov}

\DeclareMathOperator{\Pos}{Pos}
\DeclareMathOperator{\PosC}{\overline{Pos}}
\DeclareMathOperator{\Supp}{Supp}

\numberwithin{equation}{section}

\begin{document}

\title[On numerical dimensions of Calabi--Yau varieties]{On numerical dimensions of Calabi--Yau varieties}

\author{Chen Jiang}
\address{Shanghai Center for Mathematical Sciences, Fudan University, Jiangwan Campus, Shanghai, 200438, China}
\email{chenjiang@fudan.edu.cn}
\author[Long Wang]{Long Wang}
\address{Graduate School of Mathematical Sciences, The University of Tokyo, 3-8-1 Komaba, Meguro-Ku, Tokyo 153-8914, Japan}
\email{wangl11@ms.u-tokyo.ac.jp}



\begin{abstract} Let $X$ be a Calabi--Yau variety of Picard number two with infinite birational automorphism group. We show that the numerical dimension $\kappa^{\mathbb{R}}_{\sigma}$ of the extremal rays of the closed movable cone of $X$ is $\dim X/2$. More generally, we investigate the relation between the two numerical dimensions $\kappa^{\mathbb{R}}_{\sigma}$ and $\kappa^{\mathbb{R}}_{\mathrm{vol}}$ for Calabi--Yau varieties. We also compute $\kappa^{\mathbb{R}}_{\sigma}$ for non-big divisors in the closed movable cone of a projective hyperk\"{a}hler manifold. 
\end{abstract}

\maketitle

\section{Introduction} 

Given a normal projective variety $X$ and a Cartier divisor $D$ on $X$, the Iitaka dimension, which measures the asymptotic growth rate of $h^0(X,mD)$, plays a fundamental role in birational geometry. It is well-known that the Iitaka dimension is not invariant under numerical equivalence of divisors. In order to overcome this defect, there are several definitions of numerical Iitaka dimension. These numerical dimensions were expected to be equivalent. Unfortunately, this anticipation was broken very recently. By studying a specific Calabi--Yau threefold (which was previously studied by Oguiso \cite[Section 6]{Og14}), Lesieutre \cite{Le22} showed that different notions of numerical dimension for a pseudo-effective $\mathbb{R}$-divisor do not coincide. More strikingly, the numerical dimension $\kappa^{\RR}_\sigma$ of this $\mathbb{R}$-divisor (to be defined below) is even not an integer.

\begin{definition}[{\cite[Definition 1]{Le22}, \cite[V.2.5]{Na04}}]\label{def 1.1}
Let $X$ be a normal projective variety and let $D$ be an $\mathbb{R}$-Cartier $\mathbb{R}$-divisor on $X$.
The numerical dimension 
\[ \kappa^{\RR}_{\sigma}(X,D) = \kappa^{\RR}_{\sigma, \sup}(X,D) \] 
is the supremum of real numbers $\ell$ such that 
\begin{equation}\label{kss}
\limsup_{m\to \infty} \frac{h^0(X,\lfloor mD\rfloor + A)}{m^{\ell}} > 0
\end{equation} 
for some ample Cartier divisor $A$ on $X$.
(If no such $\ell$ exists, take $\kappa^{\RR}_\sigma(X,D) = -\infty$.) 
Replacing $\limsup$ by $\liminf$ in \eqref{kss}, we define $\kappa^{\RR}_{\sigma, \inf}(X,D)$. 
\end{definition}

Shortly afterwards, Hoff and Stenger \cite{HS21} generalised Lesieutre's computation to any smooth Calabi--Yau threefold of Picard number $2$ whose birational automorphism group is infinite. One new ingredient is a cone theorem for these varieties due to Lazi\'{c} and Peternell \cite{LP13}, which is predicted by the Morrison--Kawamata cone conjecture (\cite{Ka97}, see also \cite{LOP18} for a survey). 

In this paper, we generalise their results to arbitrary dimension and to the singular setting following a similar strategy. 

\begin{definition} A normal projective variety $X$ is called a \emph{Calabi--Yau variety} if $X$ has at worst $\QQ$-factorial terminal singularities, its canonical divisor $K_X$ is numerically trivial, and $h^1(X, \cO_X) = 0$. 
\end{definition}

\begin{theorem}\label{main1} Let $X$ be a Calabi--Yau variety of Picard number $2$ whose birational automorphism group $\Bir(X)$ is infinite. Let $D \not\equiv 0$ be an $\mathbb{R}$-divisor on the boundary of the closed movable cone $\MovC(X)$. Then 
\[ \kappa^{\RR}_{\sigma, \sup}(X, D) = \kappa^{\RR}_{\sigma, \inf}(X, D) = \kappa^{\RR}_{\vol, \sup}(X, D) = \kappa^{\RR}_{\vol, \inf}(X, D) = \frac{\dim X}{2}. 
\]
\end{theorem}

See Section~\ref{num} for the definition of $\kappa^{\RR}_{\vol,\sup}$ and $\kappa^{\RR}_{\vol,\inf}$. In particular, in Theorem~\ref{main1}, when $\dim X$ is odd, $\kappa^{\RR}_\sigma(X, D) = \kappa^{\RR}_{\vol}(X, D)$ is not an integer. 
The following is a concrete example of a smooth Calabi--Yau variety of Picard number $2$ with infinite birational automorphism group, which was again studied by Oguiso \cite{Og18}.

\begin{example}[{\cite[Section 3]{Og18}}]\label{og} Let $N \geq 3$ be a positive integer. Let 
\[ X = F_1 \cap F_2 \cap \cdots \cap F_{N-1} \cap Q \subset \PP^N \times \PP^N 
\] 
be a general complete intersection of $N-1$ hypersurfaces $F_i$ $(1 \leq i \leq N-1)$ of bidegree $(1, 1)$ and a hypersurface $Q$ of bidegree $(2, 2)$ in $\PP^N \times \PP^N$. Then it is not hard to see that $X$ is a smooth Calabi--Yau manifold of dimension $N$ and of Picard number $2$. 

Consider the two projections $\pi_i: X \to \PP^N\, (i = 1, 2)$, which are both of degree two by the construction of $X$. Hence there are two birational involutions $\tau_i \in \Bir(X)$ corresponding to $\pi_i$. Let $f \in \Bir(X)$ be the element defined by $f = \tau_1 \circ \tau_2$. Then $f$ is of infinite order. In fact, by \cite[Lemma 3.2]{Og18}, the eigenvalues of $f^{\ast}|_{N^1(X)_{\RR}}$ are $(2N^2-1)\pm 2N\sqrt{N^2-1}$. 

When $N = 3$, $X$ is a smooth Calabi--Yau threefold (see \cite[Section 6]{Og14}). This is the example that Lesieutre studied in \cite{Le22}. 
\end{example}

In general, for a given divisor, some numerical dimensions are easier to compute than others. For example, usually the numerical dimension $\kappa^{\RR}_{\vol}$ is slightly easier to compute; see \cite[Lemma 9]{Le22} and \cite[Section 5]{Ya22} for related computations. So it would be useful to know whether different numerical dimensions coincide under certain conditions. For example, we have the following question on comparing $\kappa^{\RR}_{\sigma}$ and $\kappa^{\RR}_{\vol}$. 

\begin{question}[{\cite[Remark 8]{Le22}, \cite[Question 3.4]{CP21}}] Let $X$ be a normal projective variety and let $D$ be a pseudo-effective $\RR$-divisor on $X$. Do the equalities $\kappa^{\RR}_{\sigma, \sup}(X,D) = \kappa^{\RR}_{\vol, \sup}(X,D)$ and $\kappa^{\RR}_{\sigma, \inf}(X,D) = \kappa^{\RR}_{\vol, \inf}(X,D)$ hold? 
\end{question}

In Section~\ref{prf1.5}, we give an affirmative answer to this question for divisors in the closed movable cone of a Calabi--Yau variety with finite birational index.

\begin{definition}Let $X$ be a Calabi--Yau variety. We say that $X$ has {\it finite birational index} if
there exists a positive integer $N$ such that for any Weil divisor $D_Y$ on any minimal model $Y$ of $X$, $ND_Y$ is Cartier.
\end{definition}

\begin{theorem}\label{com} Let $X$ be a Calabi--Yau variety with finite birational index. Then for any $\RR$-divisor class $[D] \in \MovC(X)$, we have 
\[ \kappa^{\RR}_{\sigma, \sup}(X,D) = \kappa_{\vol, \sup}^\RR(X, D) \ \ \text{and} \ \ \kappa^{\RR}_{\sigma, \inf}(X,D) = \kappa_{\vol, \inf}^\RR(X, D). 
\] 
\end{theorem}

\begin{remark}
\begin{enumerate}
 
\item Here ``having finite birational index" is a natural assumption. In fact, by Lemma~\ref{lem finite N}, a Calabi--Yau variety $X$ has finite birational index if $X$ has only finitely many minimal models up to isomorphism, or if $X$ is a hyperk\"ahler manifold, or if $\dim X\leq 3$. Moreover, it is predicted by the Morrison--Kawamata cone conjecture (\cite{Ka97}) that every Calabi--Yau variety has only finitely many minimal models up to isomorphism, hence all Calabi--Yau varieties are expected to have finite birational index.

\item It is not hard to see that $\kappa^{\RR}_{\sigma, \diamond}(X,D) \leq \kappa_{\vol, \diamond}^\RR(X, D)$ for $\diamond \in \{\sup, \inf\}$ hold in general; see e.g., \cite[Proposition 3.1]{CP21}.

\end{enumerate}

\end{remark}



In Section~\ref{prf2}, we investigate the numerical dimensions of $\RR$-divisors on projective hyperk\"{a}hler manifolds. 

\begin{definition} A \textit{projective hyperk\"{a}hler manifold} is a simply connected smooth projective variety $X$ such that $H^0(X,\Omega^2_X)$ is spanned by an everywhere non-degenerate $2$-form $\sigma$.
\end{definition} 

\begin{theorem}\label{main2} Let $X$ be a projective hyperk\"{a}hler manifold and let $ D\not \equiv 0$ be a non-big $\mathbb{R}$-divisor in the closed movable cone $\MovC(X)$. Then 
\[ \kappa^{\RR}_{\sigma, \sup}(X, D) = \kappa^{\RR}_{\sigma, \inf}(X, D) = \kappa^{\RR}_{\vol, \sup}(X, D) = \kappa^{\RR}_{\vol, \inf}(X, D) = \frac{\dim X}{2}. 
\] 
\end{theorem}

Note that the dimension of a hyperk\"{a}hler manifold is always even, and therefore, in Theorem~\ref{main2}, $\kappa^{\RR}_\sigma(X, D) = \kappa^{\RR}_{\vol}(X, D)$ is always an integer.

\subsection*{Acknowledgments} The authors thank Professor Keiji Oguiso and Doctor C\'{e}cile Gachet for suggestions, discussions and comments. We are very grateful to the referees for many useful comments and suggestions. 

The first author was supported by NSFC for Innovative Research Groups (Grant No.~12121001) and National Key Research and Development Program of China (Grant No.~2020YFA0713200). The second author was supported by JSPS KAKENHI Grant Number 21J10242.
The first author is a member of LMNS, Fudan University.

\section{Preliminaries}\label{prel}

Throughout this paper, we work over the field $\CC$ of complex numbers and refer to \cite{KM98} for knowledge about singularities and the minimal model program. 

\subsection{Notation and conventions}\label{notat} 

Let $X$ be a normal projective variety. Denote by $N^1(X)$ the group of Cartier divisors on $X$ modulo numerical equivalence. This is a free abelian group of finite rank, and its rank is called the {\it Picard number} of $X$ and denoted by $\rho(X)$. Denote $N^1(X)_{\mathbb{R}} = N^1(X) \otimes \mathbb{R}$. Denote by $\Aut(X)$ the automorphism group of $X$ and denote by $\Bir(X)$ the birational automorphism group of $X$. 

An effective divisor on $X$ is called \textit{movable} if its stable base locus has codimension at least $2$. The {\it nef cone} $\Nef(X)$, the {\it effective cone} $\Eff(X)$, and the {\it movable cone} $\Mov(X)$ are cones in $N^1(X)_{\RR}$ generated by numerical equivalence classes of nef divisors, effective divisors, and movable divisors respectively. 
The {\it ample cone} $\Amp(X)$ is the interior of $\Nef(X)$.
 The {\it pseudo-effective cone} $\EffC(X)$ is the closure of $\Eff(X)$ and the {\it big cone} $\BigC(X)$ is the interior of $\Eff(X)$. Denote the {\it closed movable cone} $\MovC(X)$ to be the closure of $\Mov(X)$ and by $\Mov(X)^{\circ}$ the interior of $\Mov(X)$. 


Suppose that $X$ has terminal singularities. 
A \textit{minimal model} of $X$ is a normal projective variety $X'$ with a birational map $\alpha: X \dashrightarrow X'$ such that $\alpha^{-1}$ has no exceptional divisors, $X'$ has only $\QQ$-factorial terminal singularities, and $K_{X'}$ is nef. 
Here we remark that this definition coincides with the standard one in \cite[Definition~3.50]{KM98}. In fact, (1--4) can be checked easily and ($5^m$) follows from the fact that for any $\alpha$-exceptional divisor $E$ on $X$, $a(E, X)=0<a(E, X')$ as $X'$ is terminal, where $a(E, X)$ is the discrepancy defined by \cite[Definition~2.22]{KM98}.

Furthermore,
if $X$ itself is a minimal model, e.g., when $X$ is a Calabi--Yau variety, then $\alpha$ is {\it small}, that is, $\alpha$ is an isomorphism in
codimension one by \cite[Corollary~3.54]{KM98}, and there is a natural linear isomorphism $\alpha_{\ast}=(\alpha^{-1})^*: N^1(X)_{\RR} \to N^1(X')_{\RR}$ which preserves movable cones and effective cones. 

\subsection{Numerical dimensions}\label{num} The numerical dimension $\kappa_{\sigma}$ has already been defined in the introduction. Here we recall other numerical dimensions. 

Let $X$ be a normal projective variety of dimension $n$. Recall that the volume of an $\RR$-Cartier $\mathbb{R}$-divisor $D$ on $X$ is defined by 
\[ \vol_X(D) = \limsup_{m \to +\infty} \frac{h^0(X, \lfloor mD\rfloor)}{m^{n}/n!}. 
\]
Note that if $D$ is a nef $\RR$-Cartier $\RR$-divisor, then $\vol_X(D)=(D^n)$. It is easy to see that volumes and global sections are preserved by small birational maps.

\begin{definition}[\cite{Le13}]\label{ndv}
Let $X$ be a normal projective variety and let $D$ be a pseudo-effective $\mathbb{R}$-Cartier $\mathbb{R}$-divisor on $X$. Fix an ample Cartier divisor $A$ on $X$. The numerical dimension 
\[ \kappa_{\vol}^\RR(X, D) = \kappa_{\vol, \inf}^\RR(X, D)
\]
is the supremum of real numbers $\ell$ for which there exists a constant $C > 0$ satisfying
\[ \vol_X(D + tA) > Ct^{\dim X - \ell} \ \text{ for all } t > 0; 
\] 
or equivalently, 
\begin{align}\label{nvi}
\liminf_{m \to +\infty} \frac{\vol_X(mD+A)}{m^{\ell}} > 0
\end{align}
by \cite[Lemma 3.2]{CP21}. 
Replacing $\liminf$ by $\limsup$ in \eqref{nvi}, we define $\kappa_{\vol, \sup}^\RR(X, D)$. 
\end{definition}

As mentioned in \cite{CP21}, all numerical dimensions we introduced depend only on the numerical class of $D$ in $N^1(X)_\RR$ (see also Lemma~\ref{lem_mul}(2) below). We give some easy facts about numerical dimensions. 

\begin{lemma}\label{lem_mul} 
Let $X$ be a normal projective variety and let $D$ be a pseudo-effective $\mathbb{R}$-Cartier $\mathbb{R}$-divisor on $X$.
Consider $\star\in \{\vol, \sigma\}$ and $\diamond\in\{\inf, \sup\}$.

\begin{enumerate}
 \item The definition of $\kappa_{\vol, \diamond}^\RR$ is independent of the choice of $A$.
 \item For a pseudo-effective $\mathbb{R}$-Cartier $\mathbb{R}$-divisor $D'$ on $X$, $$\kappa^{\RR}_{\star, \diamond}(X, D+D')\geq \kappa^{\RR}_{\star, \diamond}(X, D).$$
 In particular, $\kappa^{\RR}_{\star, \diamond}(X, D)$ depends only on the numerical class of $D$ in $N^1(X)_\RR$. 
 
 \item For any positive integer $k$, $\kappa^{\RR}_{\star, \diamond}(X, kD)=\kappa^{\RR}_{\star, \diamond}(X, D).$
\end{enumerate}
\end{lemma}

\begin{proof} 

(1) For ample Cartier divisors $A$ and $A'$ on $X$, we can find a positive integer $m_0$ such that $m_0A-A'$ and $m_0A'-A$ are ample. 
Then by \cite[Proposition~2.2.35(i) and Example~2.2.48]{La041},
\begin{align*}
 \vol_X(mD+A)\leq {}& \vol_X(mD+m_0A')\\\leq{}& \vol_X(m_0mD+m_0A')=m_0^{\dim X}\vol_X(mD+A').
\end{align*}
Similarly, $ \vol_X(mD+A')\leq m_0^{\dim X}\vol_X(mD+A)$. This proves (1).

(2) This follows from the proof of \cite[Proposition V.2.7(1)]{Na04} by \cite[Theorem V.1.3]{Na04}. Note that the above results are for non-singular varieties, so here we briefly explain how to reduce the singular case to the non-singular case. 
By Hironaka's desingularization theorem, we may take a resolution of singularities $\pi: X'\to X$. Then 
by \cite[Theorem V.1.3]{Na04}, there exists an ample divisor $A'$ on $X'$ such that
$$
H^0(X', \lfloor \pi^*(mD') \rfloor+A')\neq 0
$$
for any positive integer $m$.
So we have  
\begin{align}\label{eq:mD'A}
H^0(X, \lfloor mD' \rfloor+\pi_*A')\neq 0. 
\end{align}
On the other hand, for an ample divisor $A$ on $X$, there exists a positive integer $m'_0$ such that $H^0(X, m'_0A-\pi_*A')\neq 0$. Combining with \eqref{eq:mD'A}, we obtain
\[
h^0(X, \lfloor m(D+D') \rfloor+(m'_0+1)A)\geq h^0(X, \lfloor mD \rfloor+A)
\]
for any positive integer $m$.
Applying this inequality to the definitions of $\kappa^{\RR}_{\star, \diamond}$, we get the first statement. 

For the second statement, if $D_1\equiv D$, then the first statement implies that 
$\kappa^{\RR}_{\star, \diamond}(X, D_1)\geq \kappa^{\RR}_{\star, \diamond}(X, D)$ and 
$\kappa^{\RR}_{\star, \diamond}(X, D)\geq \kappa^{\RR}_{\star, \diamond}(X, D_1)$ by taking $D'=\pm (D-D_1)\equiv 0$.

(3) By \cite[Proposition~2.2.35(i)]{La041}, we have 
\[ \vol_X(mkD + kA) = k^{\dim X} \vol_X(mD + A)
\]
for an ample Cartier divisor $A$ on $X$, and all positive integers $k, m$. Fixing $k$ and varying $m$, we get $\kappa^{\RR}_{\vol, \diamond}(X, kD)=\kappa^{\RR}_{\vol, \diamond}(X, D).$

By (2) we have $\kappa^{\RR}_{\sigma, \diamond}(X, kD)\geq \kappa^{\RR}_{\sigma, \diamond}(X, D)$. By definition, $\kappa^{\RR}_{\sigma, \sup}(X, kD)\leq \kappa^{\RR}_{\sigma, \sup}(X, D)$ as we are taking $\limsup$. So to conclude the proof we only need to show that $\kappa^{\RR}_{\sigma, \inf}(X, k D) \leq \kappa^{\RR}_{\sigma, \inf}(X, D)$.



Set $\ell=\kappa^{\RR}_{\sigma, \inf}(X, kD)$. Fix a sufficiently small positive real number $\varepsilon$. Then there exists an ample Cartier divisor $A$ on $X$ and a positive real number $C$ such that
\begin{align}\label{eq:m ell-varepsilon}
 h^0(X, \lfloor mkD\rfloor+A)\geq C m^{\ell-\varepsilon}
\end{align}
 for all positive integers $m$.
Write $D=\sum_{i=1}^pa_iD_i$, where $D_i$'s are $p$ distinct prime divisors. 
For a positive integer $m$, write $m=sk+r$, where $s, r$ are non-negative integers and $r\leq k-1$.
Note that 
$\lfloor mD\rfloor-\lfloor skD\rfloor=\sum^p_{i=1} b_iD_i$,
where $b_i=\lfloor ra_i\rfloor$ or $1+\lfloor ra_i\rfloor$ (in particular, $0\leq b_i\leq 1+(k-1)a_i$). We can take a positive integer $m'_1$ such that
$$
h^0(X, nD_i+m'_1A)>0
$$
for any $D_i$ and any integer $0\leq n\leq 1+(k-1)a_i$. Then we can take $m_1=pm'_1$ such that
\begin{align}\label{eq:m1A}
 h^0(X, \lfloor mD\rfloor-\lfloor skD\rfloor+m_1A)>0
\end{align}
for any positive integer $m$.
Combining \eqref{eq:m1A} with \eqref{eq:m ell-varepsilon}, we have
\begin{align*}
{}&h^0(X, \lfloor mD\rfloor+(m_1+1)A)\\
\geq {}&h^0(X, \lfloor skD\rfloor+A)\\
\geq {}&Cs^{\ell-\varepsilon} \geq \frac{C{(m-k)}^{\ell-\varepsilon}}{k^{\ell-\varepsilon}}.
\end{align*}
This implies that
$\kappa^{\RR}_{\sigma, \inf}(D)\geq {\ell-\varepsilon}$. So $\kappa^{\RR}_{\sigma, \inf}(X, D)\geq \ell=\kappa^{\RR}_{\sigma, \inf}(X, kD)$ by the arbitrarity of $\epsilon$. This completes the proof.
\end{proof}

We refer to \cite{CP21, Le22, Mc18} and references therein for more numerical dimensions, their properties, and open questions. It is yet unclear whether various numerical dimensions coincide for pseudo-effective $\QQ$-divisors. For Calabi--Yau varieties, the generalised abundance conjecture together with the cone conjecture (\cite{LOP18}) predicts that for a $\QQ$-divisor $D$ in $\MovC(X)$, its strict transform under a sequence of flops is semi-ample. Hence various numerical dimensions for $D$ should all coincide with the Kodaira dimension of $D$.

\subsection{Calabi--Yau varieties with finite birational index}

We provide a criterion on when a Calabi--Yau variety has finite birational index. 

\begin{lemma}\label{lem finite N} Let $X$ be a Calabi--Yau variety. 
Suppose that one of the following holds:
\begin{enumerate}
 \item $X$ has only finitely many minimal models up to isomorphism;
 \item $X$ is a hyperk\"ahler manifold;
 \item $\dim X \leq 3$.
\end{enumerate}
Then $X$ has finite birational index. 
\end{lemma}

\begin{proof}
If $X$ has only finitely many minimal models, then the existence of $N$ follows from \cite[Theorem 1.10]{HLQ20}.

If $\dim X\leq 2$ or $X$ is a hyperk\"ahler manifold, then this is trivial as all minimal models of $X$ are smooth (the hyperk\"ahler case follows from \cite[Corollary~1]{Nami06}).

If $\dim X=3$, take a positive integer $N$ such that $NK_X\sim 0$. Then for any minimal model $Y$ of $X$, $NK_Y\sim 0$. In particular, $NK_Y$ is Cartier. Then by \cite[Corollary~5.2]{Ka88}, $ND_Y$ is Cartier for every Weil divisor $D_Y$ on $Y$.
\end{proof}
\begin{remark}
The finiteness of minimal models of a Calabi--Yau variety is predicted by the Morrison--Kawamata cone conjecture (\cite[Conjecture 1.12]{Ka97}), which is widely open even in dimension three. We refer to \cite{CO15, HT18, LW22, Og14, Og18, Sk17, Wa22, Ya22} for examples of Calabi--Yau varieties that admit finitely many minimal models up to isomorphism. 
\end{remark}


\section{Comparing numerical dimensions on Calabi--Yau varieties}\label{prf1.5}

In this section we investigate the relation between the two numerical dimensions $\kappa^{\RR}_{\sigma}$ and $\kappa^{\RR}_{\vol}$ for Calabi--Yau varieties with finite birational index. 
In order to compare $\kappa^{\RR}_{\sigma}$ and $\kappa^{\RR}_{\vol}$, we need to compare the growth of global sections of big divisors with their volumes effectively.

First we prove an estimate for base point free and big Cartier divisors by induction.



\begin{lemma}\label{lem bpf L}
Let $n$ be a positive integer.
Then there exists a positive integer $k_n$ and a positive real number $\delta_n$ satisfying the following property. 

Let $X$ be a normal projective variety of dimension $n$ with klt singularities and $H$ a base point free Cartier divisor on $X$ such that $H$ and $H-K_X$ are nef and big.
Then for any integer $k\geq k_n$, we have $h^0(X, kH)> \delta_n k^n(H^n)$.
\end{lemma}
\begin{proof}
We do induction on $n$. 

If $n=1$, then $X$ is a smooth projective curve and $\deg H>\deg K_X=2g(X)-2$.
For $k\geq 1$, by the Riemann--Roch theorem, 
$$
h^0(X, kH)=k\deg H+1-g(X)> \frac{k}{2}\deg H.
$$
So we can take $k_1=1$ and $\delta_1=\frac{1}{2}$.

For $n>1$, as $H$ is base point free, we may assume that $H$ is a normal projective variety of dimension $n-1$ with klt singularities by the Bertini theorem (see \cite[Lemma~5.17]{KM98}). 
Note that $K_H=(K_X+H)|_H$ by the adjunction formula. So $2H|_H$ and $2H|_H-K_H$ are nef and big.
By the inductive hypothesis, for any integer $k\geq k_{n-1}$, 
\begin{align}\label{eq:3.1-1}
    h^0(H, 2kH|_H)> \delta_{n-1} k^{n-1}(2H|_H)^{n-1}=\delta_{n-1} (2k)^{n-1}(H^{n}).
\end{align}
 
Consider the short exact sequence
$$
0\to \mathcal{O}_X((k-1)H)\to \mathcal{O}_X(kH)\to \mathcal{O}_H(kH)\to 0.
$$
By the Kawamata--Viehweg vanishing theorem, $H^1(X, (k-1)H)=0$ for any integer $k\geq 2$ as $(k-1)H-K_X$ is nef and big.
So 
\begin{align}\label{eq:3.1-2}
h^0(X, kH)=h^0(X, (k-1)H)+h^0(H, kH|_H)
\end{align}
for any integer $k\geq 2$.
So if $k\geq 8k_{n-1}$, we have
\begin{align*}
 h^0(X, kH)\geq {}&\sum_{i=2}^k h^0(H, iH|_H)
 \geq \sum_{i=k_{n-1}}^{\lfloor k/2\rfloor} h^0(H, 2iH|_H)\\
 \geq {}&\sum_{i=k_{n-1}}^{\lfloor k/2\rfloor}\delta_{n-1} (2i)^{n-1}(H^{n})
 > \frac{\delta_{n-1}}{2^{n+2}n}k^n(H^{n}).
\end{align*}
Here for the first inequality we use \eqref{eq:3.1-2}, and for the third one we use \eqref{eq:3.1-1}.
For the last step, we use the fact that
\begin{align*}
 \sum_{i=k_{n-1}}^{\lfloor k/2\rfloor}(2i)^{n-1}> {}& \int_{k_{n-1}-1}^{\lfloor k/2\rfloor} (2t)^{n-1} \, \textrm{d}t\\
 = {}& \frac{1}{2n}((2\lfloor k/2\rfloor)^n-(2k_{n-1}-2)^n)\\
 \geq {}& \frac{1}{2n}\left(\left(\frac{k}{2}\right)^n-(2k_{n-1}-2)^n\right)\\
 \geq {}&\frac{1}{2n}\cdot \frac{1}{2}\cdot \left(\frac{k}{2}\right)^n.
\end{align*}
Here for the first strict inequality, we use the fact that $(2t)^{n-1}$ is a strictly increasing function in $t$ so that 
$(2i)^{n-1}>  \int_{i-1}^{i} (2t)^{n-1}\, \textrm{d}t$ for each $i$, and for the last step we just use $k \geq 8k_{n-1}$.

So we may take $k_n=8k_{n-1}$ and $\delta_{n}=\frac{\delta_{n-1}}{2^{n+2}n}$.
\end{proof}

For Calabi--Yau varieties we can get an effective estimate for movable Weil divisors.

\begin{lemma}\label{lem_h0vol} Let $X$ be a Calabi--Yau variety of dimension $n$ with finite birational index.
Then there exist positive constants $C_{0}$, $C_1 > 0$, such that for any Weil divisor $D$ whose class lies in $\Mov(X)^{\circ}$ and any positive integer $k > C_0$, we have
\[ h^0(X, kD) > C_{1} k^n \vol_X(D). 
\]
\end{lemma}
\begin{proof}
By the proof of \cite[Proposition 4.6]{Wa22}, there exists a minimal model $X'$ of $X$ with a small birational map $f: X \dashrightarrow X'$, where $X'$ is a Calabi--Yau variety, such that $D' = f_*D$ is nef and big on $X'$. 
Note that $D'$ is a Weil divisor and $X$ has finite birational index, so there exists a positive integer $N$ independent of $D$ such that $ND'$ is Cartier.
By \cite[Theorem~1.1]{Kol-ebpf}, there exists a positive integer $m_1$ depending only on $n$ such that $m_1ND'$ is base point free. By \cite[Theorem~1.1]{Bir20}, there exists a positive integer $m_2$ depending only on $n$ such that $h^0(X', kD')>0$ for any integer $k\geq m_2$.

Consider $k_n$ and $\delta_n$ as in Lemma~\ref{lem bpf L}.
For any integer $k\geq 2(m_2+k_nm_1N)$, take $s=\lfloor \frac{(k-m_2)}{m_1N} \rfloor$, then $s\geq k_n$ and $k-sm_1N\geq m_2$, so by Lemma~\ref{lem bpf L},
\begin{align*}
 h^0(X', kD')\geq{}& h^0(X', sm_1ND')> \delta_n (sm_1N)^n(D'^n)\\
 \geq{}& \delta_n(k-m_2-m_1N)^n(D'^n)\\
 \geq {}&\frac{\delta_n}{2^n}k^n(D'^n).
\end{align*}
Note that $h^0(X', kD') = h^0(X, kD)$ for any positive integer $k$ and $(D'^n) =\vol_{X'}(D')= \vol_{X}(D)$. So we may take $C_0 = 2(m_2+k_nm_1N)$ and $C_{1} = \frac{\delta_n}{2^n}$.

In the end, we remark that from the proof, the positive constants $C_{0}$ and $C_1$ depend only on $n$  and $N$ (the constant from the definition of having finite birational index). 
\end{proof}

So far we only dealt with Weil divisors, but in practice we often need to deal with $\RR$-divisors, so we need to reduce the case of $\RR$-divisors to Weil divisors.

\begin{lemma}\label{lem_rdvol} Let $X$ be a $\mathbb{Q}$-factorial normal projective variety. Fix a reduced divisor $P$ on $X$. Suppose that $P$ has $p$ irreducible components. Fix an ample Cartier divisor $A$ on $X$ with $A-pP_i$ ample for each irreducible component $P_i$ of $P$.
Then for every $\RR$-divisor $D$ with $[D] \in \MovC(X)$ and $\{D\}\leq P$, we have
	
	\begin{enumerate}
	 \item the class of $\lfloor D \rfloor + A$ is in $\Mov(X)^{\circ}$, and 
	 
	 \item $\vol_X(\lfloor D \rfloor + 2A) \geq \vol_X(D + A)$. 
	\end{enumerate}
\end{lemma}

\begin{proof} 
We may write $\{D\}=\sum_{i=1}^pe_iP_i$, where $0\leq e_i<1$. Then 
\[A-\{D\}=\left(1-\sum_{i=1}^p\frac{e_i}{p}\right)A+\sum_{i=1}^pe_i\left(\frac{1}{p}A-P_i\right)\] is ample. 
So	$\floor{D} + A = D + A - \{D\}$ and its class is in $\Mov(X)^{\circ}$. 
Note that $D + A$ is big and
$\floor{D} + 2A - (D + A)=A-\{D\}$ is ample, hence by \cite[Example 2.2.48]{La041}, 
	$\vol_X(\lfloor D \rfloor + 2A) \geq \vol_X(D + A). $
\end{proof}

\begin{proposition}\label{prop_com} Let $X$ be a Calabi--Yau variety of dimension $n$ with finite birational index. Fix a reduced divisor $P$ on $X$.
Suppose that $P$ has $p$ irreducible components. 

Then there exist positive constants $C_{0}$, $C_{1} > 0$ such that for an ample Cartier divisor $A$ on $X$ with $A-pP_i$ ample for each irreducible component $P_i$ of $P$, for every $\RR$-divisor $D$ with $[D] \in \MovC(X)$ and $\{D\}\leq P$, and for any positive integer $k > C_{0}$, we have 
\[ h^0(X, \lfloor kD \rfloor + 2kA) > C_{1} k^n\vol_X(D + A). 
\]
\end{proposition} 

\begin{proof}
We can take $C_0, C_1$ as in Lemma~\ref{lem_h0vol}. By Lemma~\ref{lem_rdvol}, the class of $\lfloor D \rfloor + A$ lies in $\Mov(X)^{\circ}$. 
We can apply Lemma~\ref{lem_h0vol} to $\lfloor D + 2A \rfloor = \lfloor D \rfloor + 2A$ to get 
	\begin{align*}
h^0(X, \lfloor kD\rfloor + 2kA)\geq {}&	h^0(X, k\lfloor D + 2A \rfloor)\\
>{}& C_{1} k^n\vol_X(\lfloor D + 2A \rfloor) \geq C_{1} k^n\vol_X(D + A)
	\end{align*}
	for any positive integer $k>C_0$, where the last inequality follows from Lemma~\ref{lem_rdvol}(2).
\end{proof}



Now we are ready to prove Theorem~\ref{com}. 

\begin{proof}[Proof of Theorem~\ref{com}] 
Fix an $\RR$-divisor $D$ whose class is in $\MovC(X)$. Take $P=\Supp D$. Suppose that $P$ has $p$ irreducible components. Fix an ample Cartier divisor $A$ on $X$ with $A-pP_i$ ample for each irreducible component $P_i$ of $P$.

By applying Proposition~\ref{prop_com} to $mD$, 
	\[ h^0(X, \lfloor kmD \rfloor + 2kA) > C_{1} k^n\vol_X(mD + A)
	\]
	for any positive integer $m$ and any integer $k > C_{0}$. We fix such a positive integer $k > C_{0}$, and vary $m$. Then by definition, 
	\begin{align*} 
		\kappa^{\RR}_{\sigma, \sup}(X, kD) \geq \kappa_{\vol, \sup}^\RR(X, D) 
	\end{align*}
	and 
	\begin{align*}
		\kappa^{\RR}_{\sigma, \inf}(X, kD) \geq \kappa_{\vol, \inf}^\RR(X, D). 
	\end{align*}
	 By Lemma~\ref{lem_mul}(3), we have 
	 \begin{align*} 
		\kappa^{\RR}_{\sigma, \sup}(X, D) \geq \kappa_{\vol, \sup}^\RR(X, D) 
	\end{align*}
	and 
	\begin{align*}
		\kappa^{\RR}_{\sigma, \inf}(X, D) \geq \kappa_{\vol, \inf}^\RR(X, D). 
	\end{align*}
	The reverse inequalities are proved by \cite[Proposition 3.1]{CP21}.
	\end{proof}

\section{Numerical dimensions on Calabi--Yau varieties of Picard number $2$}\label{prf1}

In this section we compute the numerical dimensions $\kappa^\RR_{\sigma}$ and $\kappa^\RR_{\vol}$ for the extremal rays of the closed movable cone of a Calabi--Yau variety of Picard number two with infinite birational automorphism group.

\subsection{Calabi--Yau varieties with $\rho=2$}\label{pic2} Let $X$ be a Calabi--Yau variety of Picard number two whose birational automorphism group is infinite. In this subsection, we recall the structure of $\Mov(X)$ following \cite{LP13, Og14, Zh14}. 

In order to apply \cite{Zh14}, we first explain that $\Aut^0(X) = \{1\}$, where $\Aut^0(X)$ is the connected component of identity in $\Aut(X)$. In fact, since $h^1(\cO_X) = 0$, $\Aut^0(X)$ is linear by \cite[Corollary 2.19]{Br18}. 
Recall that $X$ is terminal and $K_X$ is numerically trivial, so $\kappa(X) = 0$ by \cite[Theorem~8.2]{Ka85b} and $X$ is non-uniruled by \cite[Corollary~0.3]{BDPP13}. Suppose that $\dim \Aut^0(X) \geq 1$. Recall that each linear algebraic group of positive dimension contains a subgroup isomorphic to the additive group $\mathbb{G}_a$ or the multiplicative group $\mathbb{G}_m$. For each point of $X$, the closure of the orbit of $\mathbb{G}_a$ or $\mathbb{G}_m$ is a rational curve, which implies that $X$ is a uniruled variety, a contraction. Therefore, $\Aut^0(X) = \{1\}$.

Since $\Bir(X)$ is infinite, the two extremal rays of $\MovC(X)$ are irrational by \cite[Theorem 1.2(5)]{Zh14}. Moreover, by \cite[Theorem 1.3]{Zh14}, there is a rational polyhedral cone $\Sigma$ which is a fundamental domain for the action of $\Bir(X)$ on the effective movable cone $\MovC(X) \cap \Eff(X)$, in the sense that 
\begin{align}\label{eq:ME=Sigma} \MovC(X)\cap \Eff(X) = \bigcup_{g\in \Bir(X)} g^{\ast}\Sigma
\end{align}
and $\Sigma^{\circ} \cap (g^{\ast} \Sigma)^{\circ} = \varnothing$ unless $g^{\ast} = \id$. 
Here any $g\in \Bir(X)$ induces an action $g^*: N^1(X)_\mathbb{R}\to N^1(X)_\mathbb{R}$ as in Section~\ref{notat}, and it preserves the movable cone and the effective movable cone.

Moreover, $\Sigma \subset \Mov(X)^{\circ}$ as $\Sigma$ is rational and extremal rays of $\MovC(X)$ are irrational. So this implies that the 
right hand side of \eqref{eq:ME=Sigma} is contained in $\Mov(X)^{\circ}$, which implies that
$\MovC(X)\cap \Eff(X) \subset \Mov(X)^{\circ}$. On the other hand, $\Mov(X)^{\circ}\subset \BigC(X)\subset \Eff(X)$, so we conclude that \begin{align}\label{eq:ME=Mo}\MovC(X)\cap \Eff(X) = \Mov(X)^{\circ}.\end{align}

\begin{proposition}\label{coverNef} Let $X$ be a Calabi--Yau variety with $\rho(X)=2$ and infinite $\Bir(X)$. Then there exist only finitely many minimal models of $X$ up to isomorphism. 

Moreover, there exists a rational polyhedral cone $\Pi \subset \Mov(X)^{\circ}$ and an element $f \in \Bir(X)$, such that 

\begin{enumerate}
 \item $\Mov(X)^{\circ} = \bigcup_{k \in \ZZ} (f^{k})^{\ast}\Pi$,
 
 \item $f^{\ast}|_{N^1(X)_{\RR}}$ has spectral radius (i.e. the largest absolute value of its eigenvalues) $>1$, 
 \item $f^*$ acts on the extremal rays of $\MovC(X)$ by $\begin{pmatrix}\lambda & 0\\ 0 & \lambda^{-1}\end{pmatrix}$ for some positive real number $\lambda\neq 1$,
 and 
 \item $\Pi \subset \bigcup_{i=1}^r \phi_i^*\Nef(X_i)$, where $\phi_i\colon X\dashrightarrow X_i$ are the finitely many minimal models of $X$. 
\end{enumerate}

\end{proposition}

We note that the above rational polyhedral cone $\Pi$ is not necessarily a fundamental domain for the infinite cyclic group generated by $f$. 

\begin{proof} 
(2) By \cite[Theorem 1.4(2c)]{Zh14}, $\Bir(X)$ is almost infinite cyclic, namely, the index $I = |\Bir(X) : H|$ is finite for some infinite cyclic subgroup $H \subset \Bir(X)$. Fix a generator $f \in H$. After replacing $f$ by $f^2$ (and replacing $H$ by the cyclic group generated by $f^2$), we may assume that $f$ preserves the two extremal rays of $\MovC(X)$. As $f$ also preserves $\EffC(X)$, it preserves the two extremal rays of $\EffC(X)$. 
Then $f^{\ast}|_{N^1(X)_{\RR}}$ has spectral radius $>1$ by \cite[Corollary 1.6]{Zh14}.

(3) As $f^*$ is defined over $\mathbb{Z}$, 
$\det(f^{\ast}|_{N^1(X)_{\RR}})=\pm 1$. So $f^*$ acts on the extremal rays of $\MovC(X)$ by $\begin{pmatrix}\lambda & 0\\ 0 & \lambda^{-1}\end{pmatrix}$ for some positive real number $\lambda\neq 1$.
In particular, $\EffC(X)=\MovC(X)$ since $f^{\ast}|_{N^1(X)_{\RR}}$ can only have at most $2$ linearly independent eigenvectors. 

(1) By the proof of (2), we may write $\Bir(X) = \bigcup_{j=1}^I g_j H$, and take $\Pi$ the convex hull of $\bigcup_{j=1}^I g_j^{\ast} \Sigma$, where $\Sigma$ is a rational polyhedral fundamental domain for the action of $\Bir(X)$ on $\Mov(X)^{\circ}$ as in \cite[Theorem 1.3]{Zh14}. Then clearly $\Pi$ is a rational polyhedral cone contained in $\Mov(X)^{\circ}$ because the extremal rays of $\Mov(X)^{\circ}$ are irrational. Hence $\bigcup_{k \in \ZZ} (f^{k})^{\ast}\Pi \subset \Mov(X)^{\circ}$.
On the other hand, by \eqref{eq:ME=Sigma}, 
$\MovC(X)\cap \Eff(X) =\bigcup_{k \in \ZZ} (f^{k})^{\ast}\Pi$. Hence by \eqref{eq:ME=Mo},
$\bigcup_{k \in \ZZ} (f^{k})^{\ast}\Pi = \Mov(X)^{\circ}$.

(4) We claim that the following decomposition 
\begin{align}\label{eq:local decomposition}
 \Mov(X)^{\circ} = \Mov(X) \cap \BigC(X) = \bigcup_{(X', \phi)} \phi^{\ast}\left(\Nef(X')\right) \cap \BigC(X)
\end{align}
holds, where $\phi: X \dashrightarrow X'$ runs over all minimal models of $X$, and this is a locally finite decomposition (cf. \cite[Lemma 2.15]{FHS21}). Here the decomposition is {\it locally finite} in the sense that for any point $[D]$ in the cone, there exists an open cone $U$ containing $[D]$ such that the union in the decomposition is a finite union on $U$ (cf. \cite[Theorem~2.6]{Ka97}).

In fact, by \eqref{eq:ME=Mo} and $\Mov(X)^{\circ}\subset \BigC(X)$, we have the first equality of \eqref{eq:local decomposition}. Then the second equality holds by intersecting $\BigC(X)$ with the first inclusion in 
\cite[Proposition 4.6]{Wa22}.
To show that this decomposition is locally finite, recall that by the proof of \cite[Proposition 4.6]{Wa22}, for any $[D]\in \Mov(X)^{\circ}$, a minimal model $\phi: X \dashrightarrow X'$ such that $[D]\in \phi^{\ast}\left(\Nef(X')\right)$ is a log minimal model of $(X, \epsilon D)$ for some sufficiently small positive real number $\epsilon$. By \cite[Corollary~1.1.5]{BCHM10}, there are only finitely many such log minimal models in a neighborhood $U_0$ of $[D]\in \Mov(X)^{\circ}$. Then we may take the open cone $U$ to be the cone generated by $U_0$ and hence the decomposition is locally finite.

Note that the polyhedral cone $\Pi$ is determined by its affine hyperplane section, which is a compact set. 
Hence by the local finiteness of the decomposition \eqref{eq:local decomposition}, $\Pi$ is covered by a finite union of $\phi_i^{\ast}\Nef(X_i)$ for finitely many minimal models $X_i$ with $\phi_i: X\dashrightarrow X_i$ ($1\leq i\leq r$). 

Finally, in order to show the finiteness of minimal models, it is enough to show that each minimal model $X'$ with $\phi: X \dashrightarrow X'$ is isomorphic to some $X_i$.
By (1), $(f^{k})^{\ast}\Pi \cap \phi^{\ast}\Amp(X') \neq \varnothing$ for some $k \in \ZZ$. 
Thus, $(f^k)^*\phi_i^{\ast}\Amp(X_i)\cap \phi^{\ast}\Amp(X') \neq \varnothing$ for some $i$. By \cite[Lemma 1.5]{Ka97}, $X'$ and $X_i$ are isomorphic. This completes the proof. 
\end{proof}

\subsection{Computation of numerical dimensions on Calabi--Yau varieties}

Now we are prepared to compute the numerical dimensions $\kappa^\RR_{\sigma}$ and $\kappa^\RR_{\vol}$ for the extremal rays of the closed movable cone of a Calabi--Yau variety of Picard number two with infinite birational automorphism group. 
In fact, we only need to compute $\kappa^\RR_{\vol}$ by estimating the volume of divisors that are close to the boundary of the closed movable cone using a new set of coordinates on the movable cone defined in \cite{Le22}. We will adopt the following setting in this section.

\begin{setting}\label{setting 4.1}
Let $X$ be a Calabi--Yau variety of dimension $n$ with $\rho(X)=2$ and infinite $\Bir(X)$. Fix $\RR$-divisors $R_1$ and $R_2$ which are generators of the extremal rays of $\MovC(X)$, that is,
\[ \MovC(X) = \RR_{\geq 0} [R_1] + \RR_{\geq 0} [R_2].
\]
Following \cite{Le22}, for an $\RR$-divisor $D\equiv a_1R_1+a_2R_2$ in $\text{Mov}(X)^{\circ}$,
we introduce the new set of coordinates 
\[ L_1(D) = a_1a_2 \ \ \text{and} \ \ L_2(D) = \frac{a_1}{a_2}. 
\] 
\end{setting}

Note that if $\phi: X\dashrightarrow X'$ is a minimal model of $X$, then $L_i$ are well-defined and preserved on $\text{Mov}(X')^{\circ}$ via the isomorphism $\phi_*: N^1(X)_\RR\to N^1(X')_\RR$.
Take $f \in \Bir(X)$ as in Proposition~\ref{coverNef}.
With respect to the basis $\{R_1, R_2\}$, $f$ acts on $N^1(X)_\RR$ by the matrix 
$\begin{pmatrix}
	\lambda & 0 \\ 0 & \lambda^{-1}
\end{pmatrix}, $
and we have
\[ L_1(f^{\ast} D) = L_1(D) \ \ \text{and} \ \ L_2(f^{\ast} D) = \lambda^2L_2(D). 
\]
This implies that $L_1(-)$ is a quadratic form invariant under $f^*$. Let us also notice an obvious but important fact that $L_2(-)$ describes the slope of a ray in the open cone $\Mov(X)^{\circ}$.

\begin{lemma}\label{ineqL1h0} 
Keep the notation in Setting~\ref{setting 4.1}. Then there exist positive constants $C_{11}$, $C_{21} > 0$, such that for any $\mathbb{R}$-divisor class $[D]\in \Mov(X)^{\circ}$, we have 
\[ C_{11} L_1(D)^{n/2} < \vol_X(D) < C_{21} L_1(D)^{n/2}. 
\] 
\end{lemma}	

\begin{proof} Take $\Pi\subset \Mov(X)^{\circ}$ and $f\in \Bir(X)$ as in Proposition~\ref{coverNef}. 
Fix a minimal model $\phi_i: X\dashrightarrow X_i$ as in Proposition~\ref{coverNef}(4). Fix an $\RR$-divisor class $[D_i]\in \phi_i^*\Nef(X_i)\cap \Pi$. Write $D_i \equiv a_1 R_{1} + a_2 R_ {2}$ with $a_1, a_2 > 0$. Denote by $R_{1, i}, R_{2, i}$ the strict transform of $R_1, R_2$ in $N^1(X_i)_\RR$. Then $\phi_{i*}D_i \equiv a_1 R_{1, i} + a_2 R_ {2, i}$ is nef and big on $X_i$.
We have 
\begin{align*}
	\vol_X(D_i)=(\phi_{i*}D_i)^n &= (a_1 R_{1, i} + a_2 R_{2, i})^n = \sum^n_{j=0} \binom{n}{j} a_1^j a_2^{n-j} R_{1, i}^j R_{2,i}^{n-j} \\
	&= L_1(D_i)^{n/2} \left( \sum^n_{j=0} \binom{n}{j} L_2(D_i)^{j-n/2} R_{1, i}^j R_{2, i}^{n-j} \right).
\end{align*}
Set
\[ \lambda_i(D_i) = \sum^n_{j=0} \binom{n}{j} L_2(D_i)^{j-n/2} R_{1, i}^j R_{2, i}^{n-j}. 
\]
Then
\[ \vol_X(D_i) = L_1(D_i)^{n/2} \lambda_i(D_i). 
\]
Note that $\vol_X(D_i) > 0$ and $L_1(D_i) > 0$, so $\lambda_i(D_i)$ is a positive continuous function in terms of $L_2(D_i)$ for $[D_i]\in \phi_i^*\Nef(X_i)\cap \Pi$. 
On the other hand, 
denote by $H_{1,i}$ and $H_{2,i}$ the generators of extremal rays of $\phi_i^*\Nef(X_i)\cap \Pi$. Then as $\Pi\subset \Mov(X)^\circ$, for any $[D_i]\in \phi_i^*\Nef(X_i)\cap \Pi$, possibly switching $H_{1,i}$ and $H_{2,i}$, we have
\begin{align*}
 0 < L_2(H_{1,i}) \leq L_2(D_i) \leq L_2(H_{2,i}) < +\infty.
\end{align*} 
So there exist positive constants $C_{11, i}$, $C_{21, i} > 0$, such that 
\begin{align}
 C_{11, i} < \lambda_i(D_i)=\frac{\vol_X(D_i)}{ L_1(D_i)^{n/2}} < C_{21, i} \label{eq lambdai bdd}
\end{align} 
for any $[D_i]\in \phi_i^*\Nef(X_i)\cap \Pi$.

Now for any $[D]\in \Mov(X)^\circ$,
by Proposition~\ref{coverNef}(1)(4), there exists $k \in \ZZ$ and a minimal model $\phi_i: X\dashrightarrow X_i$ such that $D\equiv (f^{k})^{\ast} D_i$ for some $\mathbb{R}$-divisor class $[D_i] \in \phi_i^*\Nef(X_i)\cap \Pi $. Note that $L_1(D) = L_1(D_i)$ and $\vol_X(D) = \vol_{X}(D_i)$.
So by \eqref{eq lambdai bdd}, 
\begin{align*}
 C_{11, i} < \frac{\vol_X(D)}{ L_1(D)^{n/2}} < C_{21, i}.
\end{align*} 
So we can find desired positive constants $C_{11} = \min\{ C_{11, i} \}$ and $C_{21} = \max\{ C_{21, i} \}$. 
\end{proof}

\begin{remark} The constants $C_{11}$ and $C_{21}$ here depend also on the ``position'' of $\Pi$ in $\Mov(X)^{\circ}$. The argument is unavailable if we only assume the finiteness of minimal models. 
\end{remark}

\begin{lemma}[{\cite[Lemma 7]{Le22}}]\label{rounddown} Keep the notation in Setting~\ref{setting 4.1}. Then there exist positive constants $C_{12}$, $C_{22}$, $C_{2} > 0$, such that for any non-zero $\RR$-divisor $D = a_1 R_1 + a_2 R_2$ with $a_1, a_2\geq 0$ and any ample Cartier divisor $A \equiv b_1 R_1 + b_2 R_2$ with $b_1$, $b_2 > C_2$, we have
	\begin{enumerate}
	 \item the class of $\lfloor D \rfloor + A $ lies in $ \Mov(X)^{\circ}$, and 
	 
	 \item $C_{12}L_1(D+A) < L_1(\lfloor D\rfloor +A) < C_{22}L_1(D + A)$. 
	\end{enumerate}
\end{lemma}

\begin{proof} Suppose that $\floor{D} \equiv \widetilde{a}_1 R_1 + \widetilde{a}_2 R_2$. 
As the support of $D-\floor{D}$ is contained in $\Supp(R_1)\cup \Supp(R_2)$ with coefficients less than one, there is a constant $c$ independent of $D$, such that $\abs{a_i - \widetilde{a}_i } < c$. 

In fact, write $\Supp(R_1)\cup \Supp(R_2)=\cup_{i}P_i$, where each $P_i$ is a prime divisor, and write $P_i\equiv p_i R_1 + q_i R_2$.
Then \[(a_1-\widetilde{a}_1) R_1 + (a_2-\widetilde{a}_2) R_2\equiv D-\floor{D}=\sum_i e_iP_i\equiv \sum_i e_i(p_i R_1 + q_i R_2) \] for some $0\leq e_i<1$. So we may just take $c= \sum_i |p_i|+ \sum_i |q_i|$.


From 
\[ D + A \equiv (a_1 + b_1) R_1 + (a_2 + b_2) R_2 
\] 
and 
\[ \floor{D} + A \equiv (\widetilde{a}_1 + b_1) R_1 + (\widetilde{a}_2 + b_2) R_2, 
\]
it is clear that
\[ \frac{L_1(\floor{D}+A)}{L_1(D+A)} = \frac{(\widetilde{a}_1+b_1)(\widetilde{a}_2+b_2)}{(a_1+b_1)(a_2+b_2)} =
\frac{\widetilde{a}_1+b_1}{a_1+b_1} \cdot \frac{\widetilde{a}_2+b_2}{a_2+b_2}.
\]
Take $C_2 = 1 + c > 1$. For $i=1, 2$, if $b_i > C_2$, then clearly $a_i + b_i > 1$, and 
\[ \widetilde{a}_i + b_i > a_i + b_i - c > 1. 
\] 
Hence $[\lfloor D\rfloor +A] \in \Mov(X)^{\circ}$. Moreover, 
\[ \frac{\widetilde{a}_i+b_i}{a_i+b_i} = 1+ \frac{\widetilde{a}_i-a_i}{a_i+b_i} \leq 1+ \frac{|\widetilde{a}_i-a_i|}{a_i+b_i} < C_2. 
\]
Similarly, 
\[ \frac{a_i+b_i}{\widetilde{a}_i+b_i} < C_2. 
\]
Therefore, 
\[ \frac{1}{C_2^2} < \frac{\widetilde{a}_1+b_1}{a_1+b_1} \cdot \frac{\widetilde{a}_2+b_2}{a_2+b_2} < C_2^2, 
\]
and we can take $C_{12} = C_2^{-2}$ and $C_{22} = C_2^2$. 
\end{proof}

\begin{proposition}\label{semi_final} Keep the notation in Setting~\ref{setting 4.1}. Then there exist positive constants $C_{13}$, $C_{23} > 0$ and an ample Cartier divisor $A$ on $X$, such that for any sufficiently large integer $m$, we have 
\[ C_{13} m^{n/2} < \vol_X(mR_i + A) < C_{23} m^{n/2}
\]
for $i=1,2$.
\end{proposition}

\begin{proof} 
It suffices to show the statement for $R_1$.
By Lemma~\ref{rounddown}, we can find positive constants $C_{12}$, $C_{22}$, $C_{2} > 0$, such that for any positive integer $m$, and any ample Cartier divisor $A \equiv b_1 R_1 + b_2 R_2$ with $b_1, b_2 > C_{2}$, we have 
\begin{align*}
C_{12}L_1(mR_1 + A) < L_1(\lfloor mR_1\rfloor + A) < C_{22}L_1(mR_1 + A). 
\end{align*}
We fix such an ample Cartier divisor $A$ on $X$ satisfying an additional property that $\lfloor mR_1\rfloor + A - mR_1$ is ample for all positive integers $m$. Note also that $mR_1 - \lfloor mR_1\rfloor\geq 0$. Thus, by \cite[Example 2.2.48]{La041}, 
\begin{align*} \vol_X(\lfloor mR_1\rfloor + A) \leq \vol_X(mR_1 + A) \leq \vol_X(\lfloor mR_1\rfloor + 2A). 
\end{align*}
It is easy to compute 
\begin{align*}
L_1(mR_1 + A) = (m + b_1)b_2 \ \ \text{and} \ \ L_1(mR_1 + 2A) = 2(m + 2b_1)b_2.
\end{align*}
Thus, combining with Lemma~\ref{ineqL1h0} for $D = \lfloor mR_1 \rfloor + A$ and $\lfloor mR_1 \rfloor + 2A$, we obtain that for any positive integer $m$, 
\[ C_{11} ( C_{12}(m + b_1)b_2 )^{n/2} < \vol_X(mR_1 + A) < C_{21} ( 2C_{22}(m + 2b_1)b_2 )^{n/2}. 
\]
Therefore, the proposition follows for any sufficiently large $m$. 
\end{proof}

\begin{proof}[Proof of Theorem~\ref{main1}] We may assume that $D=R_1$ or $R_2$. The equalities $$\kappa^{\RR}_{\vol, \sup}(X, D) = \kappa^{\RR}_{\vol, \inf}(X, D) = \frac{\dim X}{2}$$ follow immediately from Proposition~\ref{semi_final}. Combining with Theorem~\ref{com}, Proposition~\ref{coverNef}, and Lemma~\ref{lem finite N}, we complete the proof. 
\end{proof}


\section{Numerical dimensions on projective hyperk\"{a}hler manifolds}\label{prf2}
In this section we study numerical dimensions on projective hyperk\"{a}hler manifolds and give the proof of Theorem~\ref{main2}.

\subsection{Hyperk\"{a}hler manifolds} In this subsection, we recall some facts about hyperk\"{a}hler manifolds, which are also known as irreducible holomorphic symplectic manifolds. We refer to \cite{Huy99} or \cite{GHJ03} for basic properties. Let $X$ be a \textit{projective hyperk\"{a}hler manifold}, that is, a simply connected smooth projective variety such that $H^0(X, \Omega^2_X)$ is spanned by an everywhere non-degenerate $2$-form $\sigma$. It is clear that the dimension of $X$ is always even, say $\dim X = 2d$. There exists a quadratic form $q_X : H^2(X, \RR) \to \RR$ and a constant $c_X \in \QQ_{>0}$, such that for all $\alpha \in H^2(X, \RR)$, we have 
\[ \alpha^{2d} = c_X q_X(\alpha)^d. 
\] 
The above equation determines $c_X$ and $q_X$ uniquely if we assume that $q_X$ is a primitive integral quadratic form on $H^2(X, \ZZ)$ and $q_X(\sigma + \overline{\sigma}) > 0$. Here $q_X$ and $c_X$ are called the \textit{Beauville--Bogomolov--Fujiki form} and the \textit{Fujiki constant} of $X$ respectively (see \cite[Section 23.4]{GHJ03}). We will also use $q_X(-, -)$ to denote the bilinear form associated with this quadratic form. Usually we will not distinguish between $\mathbb{R}$-divisor classes and their first chern classes.


A birational map between two projective hyperk\"{a}hler manifolds is an isomorphism in codimension one, and preserves the Beauville--Bogomolov--Fujiki forms by \cite[Lemma~2.6]{Huy99}. 
Note that two birationally equivalent hyperk\"{a}hler manifolds are deformation equivalent by \cite[Theorem~4.6]{Huy99}. In particular, they admit the same Fujiki constant. 


Let $\Pos(X)$ be the \textit{positive cone} of $X$, that is, the connected component of $$\{ \alpha \in H^{1,1}(X, \RR) \mid q_X(\alpha) > 0 \}$$ containing K\"{a}hler classes. Then $\MovC(X)$ is contained in $\PosC(X)$, the closure of $\Pos(X)$, by \cite[Theorem 7]{HT09}. Moreover, for each $\RR$-divisor class $[D]\in \Mov(X)^{\circ}$, there exists a small birational map $f: X \dashrightarrow X'$, where $X'$ is a projective hyperk\"{a}hler manifold, such that $f_{\ast}D$ is nef and big on $X'$ by \cite[ Proposition 17]{HT09}.

 \subsection{Computation of numerical dimensions on hyperk\"{a}hler manifolds} We study the numerical dimensions of $\RR$-divisors on projective hyperk\"{a}hler manifolds by using Beauville--Bogomolov--Fujiki forms.

\begin{lemma}\label{lem qD}
Let $X$ be a projective hyperk\"{a}hler manifold of dimension $2d$ and fix an $\mathbb{R}$-divisor class $[D]\in \MovC(X)$. 
\begin{enumerate}

\item Then $q_X(D)\geq 0$.

\item If $D$ is not big, then $q_X(D)=0$.

\item If $[D]\in \Mov(X)^{\circ}$, then $\vol_X(D)=c_Xq_X(D)^d$. 

\item If $q_X(D, A)=0$ for some ample Cartier divisor $A$ on $X$, then $D\equiv 0$.
\end{enumerate}

\begin{proof}
(1) This directly follows from $\MovC(X)\subset \PosC(X)$ by \cite[Theorem 7]{HT09}.

(2) By \cite[Corollary~3.10]{Huy99}, if $q_X(D)>0$, then $D$ is big. 

(3) Since $[D] \in \Mov(X)^{\circ}$, by \cite[Proposition 17]{HT09}, there exists a small birational map $f: X \dashrightarrow X'$, where $X'$ is a projective hyperk\"{a}hler manifold, such that $D'=f_{\ast}D$ is nef and big on $X'$. Then 
$$
\vol_X(D)=(D'^{2d})=c_{X'}q_{X'}(D')^d=c_{X}q_{X}(D)^d.
$$

(4) This follows from (1) and the Hodge index theorem of $q_X$ by \cite[1.10]{Huy99}.
\end{proof}

\end{lemma}

\begin{proof}[Proof of Theorem~\ref{main2}]
For any positive integer $m$, $[mD+A]\in \Mov(X)^\circ$. So by Lemma~\ref{lem qD}(2)(3),
\begin{align*}
 \vol_X(mD+A)=c_Xq_X(mD+A)^d=c_X(2mq_X(D, A)+q_X(A))^d.
\end{align*}
By Lemma~\ref{lem qD}(4), $q_X(D, A)\neq 0$, so by definition,
$$
\kappa^\mathbb{R}_{\vol, \sup}(X, D)=\kappa^\mathbb{R}_{\vol, \inf}(X, D)=d=\frac{\dim X}{2}.
$$
By Theorem~\ref{com} and Lemma~\ref{lem finite N},
\[
\kappa^\mathbb{R}_{\sigma, \sup}(X, D)=\kappa^\mathbb{R}_{\sigma, \inf}(X, D)=\frac{\dim X}{2}.\qedhere
\]
\end{proof}

\begin{remark}
In the proof of Theorem~\ref{main2}, one may also use the positivity of Riemann--Roch polynomials in \cite{Ji20} to estimate global sections in order to compute $\kappa^\mathbb{R}_{\sigma, \sup}(X, D)$ and $\kappa^\mathbb{R}_{\sigma, \inf}(X, D)$ directly.
\end{remark}
\bibliographystyle{acm}
\bibliography{Numerical_dimensions_Calabi_Yau}

\end{document}